\documentclass[a4paper,12pt,leqno]{amsart}
\usepackage{latexsym}
\usepackage{amsmath}
\usepackage{amssymb}
\usepackage{amsthm}
\usepackage{indentfirst}
\usepackage[english]{babel} 
\usepackage[T1]{fontenc}
\usepackage{breqn}
\usepackage{cite}    
\usepackage{mathrsfs}                   
\usepackage[colorlinks=true]{hyperref}

\theoremstyle{plain} 
\newtheorem{thm}{Theorem}[section] 
\newtheorem{cor}[thm]{Corollary} 
\newtheorem{lem}[thm]{Lemma} 
\newtheorem{prop}[thm]{Proposition} 
\newtheorem{rmk}[thm]{Remark}

\theoremstyle{definition} 
 
\newcommand{\eps}{\varepsilon}

\numberwithin{equation}{section}

\author[D.~Mastrostefano]{Daniele Mastrostefano}

\address{University of Warwick, Mathematics Institute, Zeeman Building, Coventry, CV4 7AL, UK}
\email{Daniele.Mastrostefano@warwick.ac.uk}


\begin{document}

\title[Maximal random product sets]
      {On maximal product sets of random sets}\thanks{The author is funded by a Departmental Award and by an EPSRC Doctoral Training Partnership Award. The present work has been conducted when the author was a second year PhD student at the University of Warwick.}
  
\begin{abstract}
For every positive integer N and every $\alpha\in [0,1)$, let $B(N, \alpha)$ denote the probabilistic
model in which a random set $A\subset \{1,\dots,N\}$ is constructed by choosing independently
every element of $\{1,\dots,N\}$ with probability $\alpha$. We prove that, as $N\longrightarrow +\infty$, for every $A$ in $B(N, \alpha)$ we have $|AA|\ \sim |A|^2/2$ with probability $1-o(1)$, if and only if
$$\frac{\log(\alpha^2(\log N)^{\log 4-1})}{\sqrt{\log\log N}}\longrightarrow-\infty.$$
This improves a theorem of Cilleruelo, Ramana and Ramar\'e, who proved the above asymptotic between $|AA|$ and $|A|^2/2$ when $\alpha=o(1/\sqrt{\log N})$, and supplies a complete characterization of maximal product sets of random sets.
\end{abstract} 
   
\maketitle
\section{Introduction}
For every positive integer $N$ we indicate with $[N]:=\{1,\dots, N\}$ the set of all positive integers between $1$ and $N$. For every $\alpha\in [0,1)$, let $B(N, \alpha)$ denote the probabilistic
model in which a random set $A\subset [N]$ is constructed by choosing independently
every element of $[N]$ with probability $\alpha$.

We can interpret the random variable $|A| = \sum_{1\leq i\leq N}\textbf{1}_{i\in A}$ as a random variable with binomial distribution $\textrm{Bi}(N,\alpha)$. From this it follows that 
\begin{itemize}
\item $\mathbb{E}[|A|] =N\alpha;$
\item $\textrm{Var}(|A|) = N\alpha (1-\alpha);$
\item $\mathbb{E}[|A|^2] = (N\alpha)^2 + N\alpha(1-\alpha);$
\item $\textrm{Var}(|A|^2) = 4N^3\alpha^3(1-\alpha) + O(N^2\alpha^2);$
\item $\mathbb{E}[|A|^4] = N^4\alpha^4 + 6N^3\alpha^3(1-\alpha) + O(N^2\alpha^2).$
\end{itemize}
For an easy direct proof of the above equalities see the paper of Cilleruelo, Ramana and Ramar\'e \cite{CRR}.
In particular, it follows that
\begin{equation}
\label{asympmeanmaxsize}
\mathbb{E}[(|A|^2+|A|)/2]=\frac{N^2\alpha^2}{2}+N\alpha-\frac{N\alpha^2}{2}=\mathbb{E}[|A|^2/2]+O(N\alpha)
\end{equation}
and when $N\alpha\longrightarrow +\infty$ that
\begin{equation}
\label{importcons}
|A|\ \sim N\alpha\ \textrm{and}\ |A|^2\ \sim (N\alpha)^2	\sim |A|^2+|A|
\end{equation}
with probability $1-o(1)$, which is the content of \cite[Lemma 3.1]{CRR}.

Here for two sequences of random variables $X_1^{(N)},X_2^{(N)}$, we say $X_1^{(N)}\sim X_2^{(N)}$ if for any $\delta>0$ and $\eps>0$ there exists $N_0=N_0(\delta,\eps)\geq 1$ such that 
$$\mathbb{P}(|X_1^{(N)}-X_2^{(N)}|\geq \delta X_2^{(N)})\leq \eps,\ \textrm{if}\ N\geq N_0.$$
In short, we may write that for any $\delta>0$ 
$$\mathbb{P}(|X_1^{(N)}-X_2^{(N)}|\geq \delta X_2^{(N)})=o_{\delta}(1),\ \textrm{as}\ N\longrightarrow+\infty$$
and we will simply indicate with $X_1,X_2$ such two sequences of random variables, thus omitting the explicit dependence on $N$.

The probabilistic model $B(N,\alpha)$ has been introduced to study the expected size of product sets 
$$AA:=\{ab: a\in A, b\in A\}$$
and quotient sets 
$$A/A:=\{a/b: a\in A, b\in A\}.$$
To any set $A$ we can associate a quantity called the \emph{multiplicative energy} of $A$, defined as
$$E(A):=\{(a,b,c,d)\in A^4: ab=cd\}.$$
In the definition of $E(A)$ we tacitly assume that each quadruple is taken once without accounting for the multiplicity coming from possible symmetries (e.g. from swapping $a$ with $b$ or $c$ with $d$). The multiplicative energy thus counts the number of ``collisions'' between elements in the product or quotient sets. 

We can always find inside $E(A)$ the set of quadruples $(a,b,a,b)$ (without the multiplicity from swapping $a$ with $b$), which we denote as the set of ``trivial solutions'' (to the equation $ab=cd$), and the complementary set of ``non-trivial solutions''. The former has always size $(|A|^2+|A|)/2$. 

When the product set $AA$ has maximal cardinality it is intuitive to expect the set of trivial solutions in $E(A)$ to be much larger than the complementary set of non-trivial ones. In other words, when the number of non-trivial solutions inside $E(A)$ is somewhat ``small'' compared to $|A|^2$ we expect few collisions on average and thus a product set $AA$ of size as large as possible. 

In practice, in order to guarantee to have a random product set of maximal size, we need to put some conditions on $\alpha$ as a function of $N$.

The interesting range of $\alpha$ to consider is when $N\alpha$ is bounded away from $0$. More precisely, we can prove the following proposition.
\begin{prop}
\label{equalitycase}
Let $A$ be a random set in $B(N,\alpha)$ and assume that $\alpha=o(1/\sqrt{N})$. Then $|AA|\ =(|A|^2+|A|)/2$ with probability $1-o(1)$. 
\end{prop}
We delay the proof of Proposition \ref{equalitycase} to section 3.
In light of the above result, from now on we assume $N\alpha\longrightarrow +\infty$. 

If we increase the value of $\alpha$ we might lose the equality stated in Proposition \ref{equalitycase}, but we could at least still have an \emph{asymptotic} equality. 
A sufficient condition to guarantee that has been given in the following result (see \cite[Theorem 1.2]{CRR}).
\begin{prop}
\label{productsetsthm}
Let $A$ be a random set in $B(N,\alpha)$. If $\alpha = o((\log N)^{-1/2})$, then we have $|AA|\ \sim |A|^2/2$ with probability $1- o(1)$.
\end{prop}
If we increase the value of $\alpha$ too much we might lose the asymptotic equality stated in Proposition \ref{productsetsthm}. To see this, we first remind of Ford's result \cite[Corollary 3]{F3} on the multiplication table problem that we report below.
\begin{prop}
\label{multtable}
Let $M(x)$ be the number of positive integers $n\leq x$ which can be written as $n = m_1m_2$ with each $m_i\leq \sqrt{x}$. Then
$$M(x)\asymp \frac{x}{(\log x)^{\delta}(\log\log x)^{3/2}}\ \ \ (x\geq 2),$$
where 
$$\delta=1-\frac{1+\log\log 2}{\log 2}=0.086071....$$
\end{prop} 
Hence, considering $AA\subset [N][N]$, we get the upper bound:
\begin{equation}
\label{multiplicationtablebound}
|AA|\ll \frac{|A|^2}{\alpha^2(\log N)^{\delta}(\log\log N)^{3/2}},
\end{equation}
from which we deduce that to have a random product set of maximal size we need 
$$\alpha\ll \frac{1}{(\log N)^{\delta/2}(\log\log N)^{3/4}}.$$
An important consequence of the above bound is that we need $\alpha\longrightarrow 0$, as $N\longrightarrow +\infty.$

From the work in \cite{CRR} it is not clear though whether the value $\alpha=o((\log N)^{-1/2})$ is sharp. Understanding asymptotics for the cardinality of product sets of random sets in $B(N,\alpha)$ could serve as a good heuristic for predicting the size of product sets of deterministic sets of a given cardinality $N\alpha$. For instance, for the set of prime numbers or shifted primes it has been proven in \cite[Theorem 1.3]{CRR} that their product set is maximal, as it happens for their corresponding random models. Also the set of shifted sums of two squares $Q_N-1:=\{a^2+b^2-1: 1\leq a\leq b\leq N\}$ has been analyzed. In this case however, in \cite[Theorem 1.5]{CRR} has been computed only the order of magnitude for the size of its product set. Moreover, by comparing it with its random counterpart, the authors expressed their feelings that the size of its product set should be asymptotic to $|Q_N-1|^2/2$, even though they were not able to prove neither the asymptotic for the deterministic product set nor that for the random one.\footnote{In particular, this last assertion does not follow from \cite[Theorem 1.2]{CRR} and it was not previously known.}

The situation for quotient sets appears to be instead much clearer. Indeed, it was proved in \cite[Theorem 1.1]{CRR} that the size of the quotient set of a random set is as large as possible as soon as $\alpha$ tends to $0$ as a function of $N$. Moreover, it can be shown that the condition $\alpha = o(1)$ cannot be removed in the quotient set case. More precisely, it has been shown by Cilleruelo and Guijarro-Ord\'o\~{n}ez\cite{CO} that when $\alpha$ is a fixed real number and $N\longrightarrow +\infty$, we have $|A/A|\ \sim c_\alpha |A|^2$ with probability $1 -o(1)$, for an explicit $c_\alpha < 1$.

Coming back to product sets, in the deterministic setting it has been raised the following question:

\emph{Is it true that whenever $A \subset \{1,\dots,N\}$ is such that $|AA|\ \sim  |A|^2/2$, as $N\longrightarrow +\infty$, then $|A| = o(N(\log N)^{-1/2})$?}

This was answered negatively by Ford \cite{F}, proving the following result.
\begin{prop}
\label{extremalproductsetthm}
Let $D > 7/2$. For each $N \geq 10$ there is a set $A \subset [N]$ of size 
$$|A| \geq \frac{N}{(\log N)^{\delta/2}(\log\log N)^D},$$
with $\delta$ as in Proposition \ref{multtable}, for which $|AA|\ \sim |A|^2/2$, as $N\longrightarrow +\infty$.
\end{prop}
The proof, as sketched in \cite{F}, goes as follows. First of all, he constructed a set $B$ with some special arithmetic properties, among which that all the elements in $B$ were integers with slightly fewer prime factors compared to their expected value. Then he established a lower bound on the size of $B$ to compare with an upper bound on the multiplicative energy $E(B)$, thus showing a lack of non-trivial solutions inside $E(B)$. Finally, he selected a thin random subset $A\subset B$ that had the desired properties, borrowing some ideas from the work in \cite{CRR}.

However, the above construction and its implications do not preclude the possibility that for a \emph{random} subset $A\subset [N]$ under the model $B(N,\alpha)$ with $|A|\neq o(N(\log N)^{-1/2})$ we still have $|AA|$ asymptotic to $|A|^2/2$ with probability $1-o(1)$, leaving open the following question: 

\emph{Is it true that the condition $\alpha=o((\log N)^{-1/2})$ in Proposition \ref{productsetsthm} is also necessary?}

This paper is aimed at negatively answering to such question.
\begin{thm}
\label{mainthm}
Let $A$ be a random set in $B(N,\alpha)$, with $\alpha\in[0,1)$. Then, we have $|AA|\ \sim |A|^2/2$ with probability $1-o(1)$, as $N\longrightarrow +\infty$, if and only if
$$\frac{\log(\alpha^2(\log N)^{\log 4-1})}{\sqrt{\log\log N}}\longrightarrow -\infty.$$
\end{thm}
\begin{rmk}
In particular, closing a gap present in \cite{CRR}, for sets $A\in B(N,\alpha)$ with $\alpha\asymp 1/\sqrt{\log N}$, as for possible random models of the set of shifted sums of two squares $Q_N-1$, it follows that $|AA|\ \sim |A|^2/2$, with probability $1-o(1)$.
\end{rmk}
\section{Notations and preliminaries}
\subsection{Notations}
For a couple of real functions $f(x), g(x)$, with $g(x)>0$, we indicate with $f(x)=O(g(x))$ or $f(x)\ll g(x)$ that there exists an absolute constant $C>0$ such that $|f(x)|\leq Cg(x)$, for $x$ sufficiently large. When the implicit constant $C$ depends on a parameter $\alpha$ we instead write $f(x)\ll_{\alpha} g(x)$ or equivalently $f(x)=O_{\alpha}(g(x))$. Similarly, for a positive function $f(x)$ we say $f(x)\gg g(x)$ when instead there exists an absolute constant $c>0$ such that $|g(x)|\leq c f(x)$, for $x$ sufficiently large. Finally, when they both simultaneously hold we write $f(x)\asymp g(x)$.

Throughout, the letter $p$ is reserved for a prime number. We write $[a,b]$ to denote the least common multiple of integers $a,b$. All the other needed notations will be introduced in place.
\subsection{Preliminaries}
We now state some basic results that will be helpful in the next sections. The first of them regards upper bounds for the average of some positive multiplicative functions.
\begin{lem}
\label{Rankinestimate0}
Let $f$ be a non-negative multiplicative function. Suppose that $C$ is a constant such that
\begin{equation}
\label{sumoverprime}
\sum_{p\leq x}f(p)\log p\leq Cx
\end{equation}
for all $x\geq 1$ and that
\begin{equation}
\label{sumoverprimepowers}
\sum_{\substack{p^k:\\ k\geq 2}}\frac{f(p^k)k\log p}{p^k}\leq C.
\end{equation}
Then for $x\geq 2$,
$$\sum_{n\leq x}f(n)\ll (C+1)\frac{x}{\log x}\sum_{n\leq x}\frac{f(n)}{n}.$$
Moreover, for any positive multiplicative function $f(n)$ we also have
$$\sum_{n\leq x}\frac{f(n)}{n}\leq \prod_{p\leq x}\bigg(1+\frac{f(p)}{p}+\frac{f(p^2)}{p^2}+\cdots \bigg).$$
\end{lem}
\begin{proof}
This is \cite[ch. III, Theorem 3.5]{T}.
\end{proof}
In particular, we will need the following corollary.
\begin{cor}
\label{boundydiv}
Let $\Omega(n)$ be the number of prime factors of $n$ counted with multiplicity. For any fixed $0.1<y<1.9$ we have the uniform bound
$$\sum_{n\leq x}y^{\Omega(n)}\ll x(\log x)^{y-1}\ \ \ (x\geq 2),$$
with a uniformly bounded implicit constant.

Furthermore, if $\Omega_2(n)$ is the function which counts the number of prime factors of $n$ different from $2$ and counted with multiplicity, we have
$$\sum_{n\leq x}2^{\Omega_2(n)}\ll x\log x\ \ \ (x\geq 2).$$
\end{cor}
\begin{proof}
The first part is a special case of \cite[ch. III, Theorem 3.7]{T}, whereas the second part immediately follows from the quoted result by slightly adapting its proof.
\end{proof}
The next lemma is about some useful inequalities between the exponential function and truncations of its Taylor series expansion.
\begin{lem}
\label{Taylorexpexp}
Let 
$$T_n(x)=1+x+\frac{x^2}{2}+\cdots+\frac{x^n}{n!}$$
be the Taylor series for $\exp(x)$ at $0$ truncated after $n$ terms. Then for $x > 0$ we have
$$\exp(x)>T_n(x).$$
On the other hand, for $x<0$, we have
$$\left\{ \begin{array}{ll}
        \exp(x)>T_n(x) & \mbox{if $n$ odd};\\
        \exp(x)<T_n(x) & \mbox{if $n$ even}.\end{array} \right.$$ 
\end{lem}
\begin{proof}
By the Taylor expansion of the exponential at $0$ with the Lagrange remainder, we have:
\begin{align*}
\exp(x)=T_n(x)+\frac{\exp(\xi)}{(n+1)!}x^{n+1},
\end{align*}
for a certain $\xi$ between $0$ and $x$. Since $\exp(\xi)\geq 0$ we immediately deduce the thesis.
\end{proof}
We conclude this section by proving that if two sequences of positive random variables are asymptotic, and we have some control on the second moment of at least one of them, then their mean values will also be asymptotic. We explain this in details in the following lemma, in which the particular case of $|AA|$ and $(|A|^2+|A|)/2$ has been analysed.
\begin{lem}
\label{propasympmean}
As $N\alpha\longrightarrow +\infty$, if $|AA|\ \sim (|A|^2+|A|)/2$, with probability $1-o(1)$, we have 
$$\mathbb{E}[|AA|]\sim \mathbb{E}[(|A|^2+|A|)/2],\ \ \ \text{as $N\longrightarrow +\infty$}.$$
\end{lem}
\begin{proof} 
To simplify notations let us put $X_1=(|A|^2+|A|)/2$ and $X_2=|AA|$.
We certainly have
$$\mathbb{E}[X_1]=\mathbb{E}[X_1-X_2]+\mathbb{E}[X_2]$$
where the first mean value on the right hand side above is for any $\eps>0$
\begin{align*}
&=\mathbb{E}[(X_1-X_2)\textbf{1}_{(X_1-X_2)\geq\ \eps X_2}]+\mathbb{E}[(X_1-X_2)\textbf{1}_{(X_1-X_2)\leq\ \eps X_2}]\\
&\leq \mathbb{E}[(X_1-X_2)\textbf{1}_{(X_1-X_2)\geq\ \eps X_2}]+\eps \mathbb{E}[X_2]\\
&\leq \sqrt{\mathbb{E}[(X_1-X_2)^2]\mathbb{P}(X_1-X_2\geq \eps X_2)}+\eps \mathbb{E}[X_2]&\text{(by Cauchy--Schwarz)}\\
&\leq o_{\eps}\bigg(\sqrt{\mathbb{E}[X_1^2+X_2^2]}\bigg)+\eps \mathbb{E}[X_2]&\text{(by hypothesis)}\\
&\leq o_{\eps}\bigg(\sqrt{\mathbb{E}[X_1^2]}\bigg)+\eps \mathbb{E}[X_2]&\text{(since $X_2\leq X_1$)}.\\
\end{align*}
Using \eqref{importcons} and the asymptotics on the moments of $|A|$, it is immediate to show that
$\mathbb{E}[X_1^2]\sim \mathbb{E}[X_1]^2$. Putting the above estimates together, we deduce that
$$\mathbb{E}[X_1](1-o_\eps(1))\leq \mathbb{E}[X_2](1+\eps).$$
From this we can reach the conclusion. Indeed, choose $N_0=N_0(\eps)$ such that $o_{\eps}(1)\leq \eps$, for any $N\geq N_0$. Then 
$$\bigg|\frac{\mathbb{E}[X_1]}{\mathbb{E}[X_2]}-1\bigg|\leq 2\eps+O(\eps^2),$$
for any $N\geq N_0$, from which the thesis.
\end{proof}
\section{Proof of the introductory results}
In this section we give a proof of the first two propositions stated in the introduction.
\begin{proof}[Proof of Proposition \ref{equalitycase}]
Every element in $AA$ is by definition a product $ab$, with $a,b\in A$. The number of such products, without accounting for the multiplicity coming from the symmetry $ab=ba$, is at most $(|A|^2+|A|)/2.$ We will now show that the probability of having $|AA|\ =(|A|^2+|A|)/2$ tends to $1$. Equivalently, if we let 
$$\Sigma(A):=\frac{|A|^2+|A|}{2}-|AA|$$
we will show that 
$$\mathbb{P}(\Sigma(A) \geq 1)=o(1).$$
To this aim we introduce the following notation. We indicate with 
$$\tau_N(n):=\#\{(j,k)\in [N]\times [N]: n=jk\}$$ 
the number of representations of a positive integer $n$ as product $n = jk$, with $1\leq j, k\leq  N$. Clearly, we have 
\begin{equation}
\label{averagetau0}
\sum_{1\leq n\leq N^2}\tau_N(n)= N^2.
\end{equation}
Hence, we can infer that
\begin{align*}
\mathbb{P}(\Sigma(A)\geq 1)=\mathbb{P}(\exists (a,b,c,d)\in A^4: ab=cd\ \text{and}\ a\neq c,d)
&\leq\alpha^4\sum_{ab\in [N][N]}\sum_{\substack{d|ab\\ d\neq a,b\\ d\leq N,ab/d\leq N}}1\\
&\leq \alpha^4\sum_{ab\in [N][N]}\tau_N(ab)\\
&\leq\alpha^4\sum_{n\leq N^2}\tau_N(n)=\alpha^4N^2,
\end{align*}
by the union bound and \eqref{averagetau0}. Since by hypothesis $\alpha=o(1/\sqrt{N})$, we get $\mathbb{P}(\Sigma(A)\geq 1)=o(1)$, hence the thesis.
\end{proof}
Proposition \ref{productsetsthm} is the content of \cite[Theorem 1.2]{CRR} (for a generalization thereof to iterated product sets of random sets see instead Sanna \cite{SA}), but here we are going to present a new alternative proof.
\begin{proof}[Proof of Proposition \ref{productsetsthm}]
By an application of Cauchy--Schwarz's inequality, we have
\begin{equation}
\label{CauchySchwarz}
\bigg(\frac{|A|^2+|A|}{2}\bigg)^2=\bigg(\sum_{x\in AA}r_{AA}(x)\bigg)^2\leq |AA|\bigg(\sum_{x\in AA}r_{AA}(x)^2\bigg)=|AA|E(A),
\end{equation}
where $r_{AA}(x)$ is the number of representations of $x$ as a product of two elements in $A$, without accounting for possible symmetries. For an appearance of the use of inequality \eqref{CauchySchwarz} to produce a lower bound for the size of product sets see the Tao and Vu's textbook \cite[Lemma 2.30]{TV}.

Since $E(A)=(|A|^2+|A|)/2+R(A)$, where $R(A)$ is the number of non-trivial solutions to $ab=cd$ in $A$, from \eqref{CauchySchwarz} we get
\begin{equation}
\label{consequenceCS}
\frac{((|A|^2+|A|)/2)^2}{(|A|^2+|A|)/2+R(A)}\leq |AA|.
\end{equation}
Moreover, we have 
\begin{align}
\label{meanvalueR(A)}
\mathbb{E}[R(A)]&=\sum_{\substack{1\leq a,b,c,d\leq N\\ ab=cd\\a\neq b,c,d}}\mathbb{P}(a,b,c,d\in A)+\sum_{\substack{1\leq a,c,d\leq N\\ a^2=cd\\ a\neq c,d}}\mathbb{P}(a,c,d\in A)\\
&\leq \sum_{\substack{1\leq a,b,c,d\leq N\\ ab=cd}}\alpha^4+\sum_{1\leq a\leq N}\sum_{\substack{1\leq d\leq N\\ d|a^2}}\alpha^3\nonumber\\
&\leq \alpha^4E([N])+\alpha^3\sum_{1\leq a\leq N}\tau(a^2),\nonumber
\end{align}
where $\tau(n)$ is the divisor function, which counts the number of positive divisors of a positive integer $n$. It has been proven in \cite[Lemma 2.1]{CRR} that $E([N])\ll N^2\log N$. Moreover, we have \footnote{The correct order of magnitude for the partial sum of $\tau(n^2)$ over the positive integers $n$ up to $x$ is $x(\log x)^2$, but we do not need this degree of precision here.}
$$\sum_{n\leq x}\tau(n^2)\ll x(\log x)^3\ \ \ (x\geq 2),$$
which can be easily derived from Lemma \ref{Rankinestimate0}. We deduce that \eqref{meanvalueR(A)} is
$$\ll \alpha^4N^2\log N+\alpha^3N(\log N)^3.$$
We conclude that values of $\alpha=o((\log N)^{-1/2})$ makes the above of size $o(\alpha^2N^2)$. By Markov's inequality we then have for any $\eps>0$
$$\mathbb{P}(R(A)>\eps \alpha^2N^2)\leq \frac{\mathbb{E}[R(A)]}{\eps\alpha^2N^2}=o_\eps(1).$$
Combining this with \eqref{importcons} and \eqref{consequenceCS}, we deduce that
$$\frac{|A|^2+|A|}{2}(1+O(\eps))\leq |AA|\leq \frac{|A|^2+|A|}{2}$$
with probability $1-o_{\eps}(1)$. By the arbitrariness of $\eps>0$, we get the result.
\end{proof}
\section{Proof sketch of Theorem \ref{mainthm}}
\subsection{The basic set up}
Let us define
$$X_A:=\frac{|A|^2+|A|}{2}-|AA|\geq 0.$$
We would like to show that for any $\delta>0$, there exists $N_0=N_0(\delta)>0$ such that for any $N\geq N_0$ we have
$$\mathbb{P}(X_A\geq \delta (|A|^2+|A|)/2)=o_{\delta}(1).$$
Thanks to \eqref{importcons} it suffices to show that
$$\mathbb{P}(X_A\geq \delta (N\alpha)^2)=o_{\delta}(1)$$
and thus that 
$$\mathbb{E}[X_A]=o(\alpha^2N^2),$$
by means of Markov's inequality. In order to achieve this, we will express the mean of $X_A$ in terms of a certain average of the function $\tau_N(n)$. More precisely, by \eqref{importcons} and \eqref{averagetau0}, and since from the proof of \cite[Proposition 3.2]{CRR} we know that
\begin{equation}
\label{meanproductset}
\mathbb{E}[|AA|]=\sum_{1\leq n\leq N^2}(1-(1-\alpha^2)^{\tau_N(n)/2})+O(N\alpha),
\end{equation}
we deduce that 
\begin{align}
\label{asympmeanconj}
\mathbb{E}[X_A]\approx \sum_{1\leq n\leq N^2}\bigg(\frac{\alpha^2\tau_N(n)}{2}-1+(1-\alpha^2)^{\tau_N(n)/2}\bigg).
\end{align}
The term inside the parenthesis is the difference between the binomial $(1-\alpha^2)^{\tau_N(n)/2}$ and its first order Taylor expansion. We then split the sum into two parts: the first one being on those integers $\mathcal{S}_1\subset [N^2]$ where it is possible to Taylor expand the above binomial a little further, the second one being on the rest $\mathcal{S}_2$. 

Since then summand in \eqref{asympmeanconj} is always dominated by $\alpha^2\tau_N(n)$, we can simply bound the contribution from $\mathcal{S}_2$ with 
\begin{equation}
\label{contributionfromS2}
\ll \alpha^2\sum_{n\in\mathcal{S}_2}\tau_N(n).
\end{equation}
On the other hand, by Taylor expanding the binomial, the contribution from $\mathcal{S}_1$ is 
\begin{align}
\label{contributionfromS1}
\approx \alpha^4 \sum_{n\in\mathcal{S}_1}\tau_N(n)^2.
\end{align}
We are left with suitably defining the sets $\mathcal{S}_1, \mathcal{S}_2$ in order to make the above two sums small. It is clear that we need first to understand the distribution of the function $\tau_N$.
\subsection{Heuristic behaviour of $\tau_N$}
We claim that roughly speaking we may think of $\tau_N(n)$ as
$$\tau_N(n)\approx 2\tau(n)\bigg(1-\frac{\log n}{2\log N}\bigg)\ \ \ (\text{for most}\ n\leq N^2),$$
at least when we consider $\tau_N$ on average over a ``large'' set of integers.

Indeed, if we assume that for most positive integers $n\leq N^2$ the set $\{\log d/\log N: d|n\}$ is roughly uniformly distributed over the interval $[0,1]$, we have
\begin{align*}
\tau_N(n)=\#\{d|n: n/N\leq d\leq N\}&\approx\sum_{k=\lfloor\frac{\log(n/N)}{\log 2}\rfloor}^{\lfloor\frac{\log N}{\log 2}-1\rfloor}\sum_{\substack{d|n\\ 2^k<d\leq 2^{k+1}}}1\\
&\approx \tau(n)\sum_{k=\lfloor\frac{\log(n/N)}{\log 2}\rfloor}^{\lfloor\frac{\log N}{\log 2}-1\rfloor}\frac{\log 2}{\log N}\\
&\approx \frac{\tau(n)\log 2}{\log N}\bigg(\frac{\log N-\log(n/N)}{\log 2}\bigg)\\
&=\frac{\tau(n)}{\log N}(2\log N-\log n)\\
&=2\tau(n)\bigg(1-\frac{\log n}{2\log N}\bigg).
\end{align*}
We note that the mass of the average of $\tau(n)$ over the integers $n\leq N^2$ is mainly concentrated around those integers close, but not too much, to $N^2$. Indeed, for the $k$-th moment of $\tau(n)$ we have
\begin{equation}
\label{averagemomentsdivfct}
\sum_{n\leq N^2}\tau(n)^k\sim c_k N^2(\log N)^{2^k-1},\ \ \ \text{as $N\longrightarrow +\infty$},
\end{equation}
for a certain $c_k>0$ (see e.g. Luca and T\'{o}th's paper \cite{LT}).
We deduce that, for any $B\geq 1$, the part of the sum over $n\leq N^2/B$, say, contributes 
$$\ll \frac{N^2(\log N)^{2^k-1}}{B},$$
thus making a negligible contribution to \eqref{averagemomentsdivfct}, when $B>0$ is large enough. 

On the other hand, for the part of the sum over $n>N^{2}(1-1/C)$, we again get a negligible contribution to \eqref{averagemomentsdivfct}, when $C>0$ is large enough, by Shiu's theorem \cite[Theorem 1]{SH}.

In conclusion, the main contribution to the sum in \eqref{averagemomentsdivfct} comes from those integers $n\asymp N^2$. Therefore, we can recast our heuristic as
\begin{equation}
\label{mainheuristic}
\tau_N(n)\approx \frac{\tau(n)}{\log N}.
\end{equation}
\subsection{Heuristics for the mean of $X_A$: $\mathcal{S}_2$--part}
Using \eqref{mainheuristic} we may rewrite \eqref{contributionfromS2} as
\begin{align*}
\approx\frac{\alpha^2}{\log N}\sum_{n\in \mathcal{S}_2}\tau(n).
\end{align*}
It is well-known that the average of $\tau(n)$ is small (compared to the whole average given in \eqref{averagemomentsdivfct} for $k=1$) on those integers $n\leq N^2$ with a number of distinct prime factors $\omega(n)$ far from $2\log\log N$. More precisely, we can prove the following lemma.
\begin{lem}
\label{lemtailtau}
For any $0<\eps<1$, we have
$$\sum_{\substack{1\leq n\leq N^2\\ |\omega(n)-2\log\log N|>\eps \log\log N}}\tau(n)\ll N^2(\log N)^{1-2\eta},$$
with
$$\eta:=\bigg(1+\frac{\eps}{2}\bigg)\log\bigg(1+\frac{\eps}{2}\bigg)-\frac{\eps}{2}$$
and a uniformly bounded implicit constant.
\end{lem}
\begin{proof}
We focus on estimating only the part of the sum corresponding to integers $n\leq N^2$ for which 
$$\omega(n)>(2+\eps)\log\log N,$$
since the estimate for the complementary part can be then similarly deduced. The sum we would like to estimate can be interpreted as the mean value of the indicator function on the above condition weighted with $\tau(n)$. In analogy to the exponential moment method in probability theory, we let $y>1$ be a parameter to determine later on and upper bound the aforementioned sum with:
\begin{align*}
y^{-(2+\eps)\log\log N}\sum_{1\leq n\leq N^2}\tau(n)y^{\omega(n)}&\ll N^2 (\log N)^{2y-1}y^{-(2+\eps)\log\log N}\\
&=N^2 (\log N)^{2y-1-(2+\eps)\log y},
\end{align*}
by Lemma \ref{Rankinestimate0}, with a uniformly bounded implicit constant. Indeed, conditions \eqref{sumoverprime} and \eqref{sumoverprimepowers} are satisfied by
\begin{align*}
&\sum_{p\leq x}\tau(p)y^{\omega(p)}\log p=2y\sum_{p\leq x}\log p\\
&\sum_{\substack{p^k:\\ k\geq 2}}\frac{\tau(p^k)y^{\omega(p^k)}k\log p}{p^k}=y\sum_{\substack{p^k:\\ k\geq 2}}\frac{k(k+1)\log p}{p^k}
\end{align*}
and by Chebyshev's estimates \cite[ch. I, Corollary 2.12]{T}. Moreover, for any $x\geq 2$ we have
$$\prod_{p\leq x}\bigg(\sum_{k\geq 0}\frac{\tau(p^k)y^{\omega(p^k)}}{p^k}\bigg)=\prod_{p\leq x}\bigg(1+y\sum_{k\geq 1}\frac{k+1}{p^k}\bigg)\ll \prod_{p\leq x}\bigg(1+\frac{2y}{p}\bigg)\ll (\log x)^{2y},$$
by Mertens' formula \cite[ch. I, Theorem 1.12]{T}, with a uniformly bounded implicit constant independent of $y$.

We can now optimize in $y$. Letting $y:=1+\eps/2$ we reach the thesis, since
$$2y-1-(2+\eps)\log y=1+\eps-(2+\eps)\log\bigg(1+\frac{\eps}{2}\bigg)=1-2\eta.$$
Note that $\eta>0$, if $\eps$ small enough. Thus, the upper bound we found is non-trivial.
\end{proof}
The parameter $\eps$ in Lemma \ref{lemtailtau} has not been specified yet. On the other hand, the bound there strongly depends on it. We then need a careful choice. By working in analogy to the Turan--Kubilius' inequality (see e.g. \cite[ch. III, Theorem 3.1]{T}), we define the set $\mathcal{S}_1$ as:
$$\mathcal{S}_1:=\{n\leq N^2:|\omega(n)-2\log\log N|\leq \lambda(N)\sqrt{\log\log N}\}.$$ 
Here $\lambda(N)$ is any function with $\lambda(N)\longrightarrow +\infty$, as $N\longrightarrow +\infty$, and $\lambda(N)/\sqrt{\log\log N}\longrightarrow 0$.

We can now make the following consideration: since on a positive proportion of integers $n\leq N^2$ we may identify $\tau(n)$ with $2^{\omega(n)}$ and since on $\mathcal{S}_1$ we have $\omega(n)$ equal to $2\log\log N$ plus a smaller error term, in view of our previous heuristic \eqref{mainheuristic}, we can expect
$$\log(\tau_N(n))\approx (\log 4-1)\log\log N\ \ \ (\text{for most}\ n\leq N^2),$$
which can be considered as the ``normal'' order of $\log(\tau_N(n))$ (for a rigorous definition of the normal order of an arithmetical function, see e.g. \cite[ch. III, eq. (3.1)]{T}).

Combining this with the result of Lemma \ref{lemtailtau}, we expect \eqref{contributionfromS2} to be bounded by
\begin{equation}
\label{tailgauss}
\ll \alpha^2 N^2\exp\bigg(-\frac{\lambda(N)^2}{4}(1+o(1))\bigg),
\end{equation}
by choosing the parameter $\eps$ in Lemma \ref{lemtailtau} as $\eps:=\lambda(N)/\sqrt{\log\log N}$ so that the parameter $\eta$ in Lemma \ref{lemtailtau} equals
\begin{align*}
\eta&=\bigg(1+\frac{\lambda(N)}{2\sqrt{\log\log N}}\bigg)\log \bigg(1+\frac{\lambda(N)}{2\sqrt{\log\log N}}\bigg)-\frac{\lambda(N)}{2\sqrt{\log\log N}}\\
&=\bigg(1+\frac{\lambda(N)}{2\sqrt{\log\log N}}\bigg)\bigg(\frac{\lambda(N)}{2\sqrt{\log\log N}}-\frac{\lambda(N)^2}{8\log\log N}+O\bigg(\frac{\lambda(N)^3}{(\log\log N)^{3/2}}\bigg)\bigg)-\frac{\lambda(N)}{2\sqrt{\log\log N}}\\
&=-\frac{\lambda(N)^2}{8\log\log N}+\frac{\lambda(N)^2}{4\log\log N}+O\bigg(\frac{\lambda(N)^3}{(\log\log N)^{3/2}}\bigg)\\
&=\frac{\lambda(N)^2}{8\log\log N}(1+o(1)).\\
\end{align*}
Since $\lambda(N)\longrightarrow +\infty$, we readily see that \eqref{tailgauss} is $o(\alpha^2 N^2)$.
\subsection{The Erd\H{o}s--Kac's theorem}
We should stop a moment to understand why \eqref{tailgauss} is essentially best possible.

The origin of this stems from the distribution of the function $\omega(n)$ over the integers $n\leq N^2$. We define the probability space $([N],\mathcal{P}([N]), \mathbb{P}_N)$, where $\mathcal{P}([N])$ is the power set of $[N]$ and $\mathbb{P}_N$ denotes the discrete uniform measure on $[N]$. A classical consequence of the Turan--Kubilius' inequality (see e.g. \cite[ch. III, Theorem 3.4]{T}) is the following result.
\begin{prop}
\label{Turankubilius}
Given any function $t(N)\geq 1$, we have
$$\mathbb{P}_N(|\omega(n)-\log\log N|>t(N)\sqrt{\log\log N})\ll \frac{1}{t(N)^2}.$$
In particular, if $t(N)\longrightarrow+\infty$, as $N\longrightarrow +\infty$, then ``almost all'' numbers $n\leq N^2$ (in the sense of asymptotic density) satisfy:
$$|\omega(n)-\log\log N|\leq t(N)\sqrt{\log\log N}.$$
\end{prop}
Proposition \ref{Turankubilius} gives the feeling that we really need to work here with unbounded functions in order to get the infinitesimal order contribution necessary to show that \eqref{asympmeanconj} is $o(\alpha^2 N^2)$. However, in the proof of our main result we will need a deeper understanding of the distribution of $\omega(n)$ to just work with arbitrarily large positive constants instead of unbounded functions $\lambda(N)$. In fact, an application of the moments method leads to the following well celebrated consequence of the Erd\H{o}s--Kac's theorem (see e.g. \cite[ch. III, Theorem 4.15]{T}).
\begin{prop}
\label{Erdoskac}
Under the probability measure $\mathbb{P}_N$, we have
$$\frac{\omega(n)-\log\log N}{\sqrt{\log\log N}}\longrightarrow N(0,1)\ \ \ (\text{as $N\longrightarrow +\infty$}),$$
where $N(0,1)$ indicates a random variable of standard normal distribution. 
\end{prop}
Therefore, for any fixed $t\geq 1$ we have
\begin{equation}
\label{consequenceErdoskac}
\mathbb{P}\bigg(\frac{|\omega(n)-\log\log N|}{\sqrt{\log\log N}}>t\bigg)\longrightarrow \frac{1}{\sqrt{2\pi}}\int_{t}^{+\infty}e^{-s^2/2}ds\leq \frac{e^{-t^2/2}}{t\sqrt{2\pi}},
\end{equation}
where the last inequality follows from the fact that for any $s\geq t>0$ we have
$$\int_{t}^{+\infty}1\cdot e^{-s^2/2}ds\leq \int_{t}^{+\infty}\frac{s}{t}e^{-s^2/2}ds=\frac{e^{-t^2/2}}{t}.$$
Since we already noticed that the sum in \eqref{contributionfromS2} can be recast in terms of an average of the indicator function $\textbf{1}_{|\omega(n)-\log\log N|> \lambda(N)\sqrt{\log\log N}}$, weighted with $\tau(n)$, \eqref{consequenceErdoskac} gives the feeling that the bound \eqref{tailgauss} is essentially best possible here.
\subsection{Heuristics for the mean of $X_A$: $\mathcal{S}_1$--part}
Thanks to \eqref{tailgauss} we can discard the contribution of $\mathcal{S}_2$ from the mean of $X_A$ \eqref{asympmeanconj} and we are left with understanding only that coming from $\mathcal{S}_1$, or equivalently with upper bounding \eqref{contributionfromS1}. Also, notice that until now we have not needed to specify the value of $\alpha$ in order to make the sum \eqref{contributionfromS2} negligible. On the other hand, the requirement on $\alpha$ will clearly emerge from the next computations, in which we are going to heuristically work out the second moment of $\tau_N$ over $\mathcal{S}_1$.

If we indicate with
$$\pi_k(N^2):=\#\{n\leq N^2: \omega(n)=k\},$$
the number of integers $n\leq N^2$ with exactly $k$ distinct prime factors, by the definition of the set $\mathcal{S}_1$ and thanks to \eqref{mainheuristic}, we can roughly upper bound \eqref{contributionfromS1} with
\begin{equation}
\label{sumconcentrationomega}
\frac{\alpha^4}{(\log N)^2}\sum_{|k-2\log\log N|\leq \lambda(N)\sqrt{\log\log N}}4^k \pi_k(N^2),
\end{equation}
again by identifying $\tau(n)$ with $2^{\omega(n)}$. 

The classic Landau's theorem, in the form given by an application of the Selberg--Delange's method \cite[ch. II, Theorem 6.4]{T}, supplies a uniform upper bound for $\pi_k(N^2)$, when $k$ is at most a constant times $\log\log N$. We report such result below.
\begin{prop}
\label{Landauthm}
Let $A>0$. Then uniformly on $N\geq 2$ and $1\leq k\leq A\log\log N$ we have
$$\pi_k(N^2)\ll_A \frac{N^2}{\log N}\frac{(\log\log N)^{k-1}}{(k-1)!}.$$
\end{prop}
Plugging the above estimate into \eqref{sumconcentrationomega} we can upper bound this last one with
\begin{equation}
\label{sumconcentrationomega2}
\ll\frac{\alpha^4N^2}{(\log N)^3}\sum_{|k-2\log\log N|\leq \lambda(N)\sqrt{\log\log N}}\frac{(4\log\log N)^{k-1}}{(k-1)!}.
\end{equation}
We now need a sharp upper bound for the sum in \eqref{sumconcentrationomega2}. This can be deduced from Norton's bounds \cite{N}, whose special case we report next.
\begin{lem}
\label{lemexpTaylorexp}
Suppose $0\leq h<m\leq x$ and $m-h\geq \sqrt{x}$. Then 
$$\sum_{h\leq k\leq m}\frac{x^k}{k!}\asymp \min\bigg(\sqrt{x},\frac{x}{x-m}\bigg)\frac{x^m}{m!}.$$
\end{lem}
By applying the above lemma with 
\begin{align*}
&h:=2\log\log N-\lambda(N)\sqrt{\log\log N}\\
&m:=2\log\log N+\lambda(N)\sqrt{\log\log N}\\
&x:=4\log\log N
\end{align*}
and using Stirling's formula, we can upper bound the sum in \eqref{sumconcentrationomega2} with
\begin{align*}
&\ll\frac{(4e\log\log N)^m}{(2\log\log N+\lambda(N)\sqrt{\log\log N})^{m}\sqrt{m}}\\
&=\frac{(2e)^m}{\sqrt{m}}\bigg(1+\frac{\lambda(N)}{2\sqrt{\log\log N}}\bigg)^{-m}\\
&=\frac{(\log N)^{\log 4+2}\exp((\log 2+1)\lambda(N)\sqrt{\log\log N})}{\sqrt{m}}\exp\bigg(-m\log\bigg(1+\frac{\lambda(N)}{2\sqrt{\log\log N}}\bigg)\bigg)\\
&\ll \frac{(\log N)^{\log 4+2}\exp((\log 2+o(1))\lambda(N)\sqrt{\log\log N})}{\sqrt{\log\log N}}\\
&=(\log N)^{\log 4+2}\exp((\log 2+o(1))\lambda(N)\sqrt{\log\log N}).
\end{align*}
Collecting the previous estimates together, we can overall infer that we expect a contribution from \eqref{contributionfromS1} of roughly at most
\begin{equation}
\label{heuristicmeanXA}
\alpha^4 N^2 (\log N)^{\log 4-1}\exp((\log 2+o(1))\lambda(N)\sqrt{\log\log N}).
\end{equation}
Considering the arbitrariness of $\lambda(N)$, in order to make the above of size $o(\alpha^2 N^2)$ we are led to take $\alpha$ such that
$$\alpha^2=o\bigg(\frac{1}{(\log N)^{\log 4-1}\exp(\lambda(N)\sqrt{\log\log N})}\bigg).$$
In fact, we will show that the stronger condition given in the statement of Theorem \ref{mainthm} already suffices. We then expect those values of $\alpha$, by the discussion at the start of section $4$, to guarantee a corresponding random product set in $B(N,\alpha)$ of maximal size. 
\subsection{Heuristic for the necessary condition in Theorem \ref{mainthm}}
By \eqref{asympmeanconj} we can express the mean value of $X_A$ as the average of 
$$\frac{\alpha^2\tau_N (n)}{2}-1+(1-\alpha^2)^{\tau_N(n)/2}$$
which, from the considerations in the previous subsections, can be roughly seen as 
\begin{equation}
\label{heuristicsummand}
\frac{\alpha^2t(N)}{2}-1+(1-\alpha^2)^{t(N)/2},
\end{equation}
on the set $\mathcal{S}_1$, where for any function $\xi(N)\longrightarrow +\infty$ we define
$$t(N)=t_{\xi}(N):=(\log N)^{\log4 -1}\exp(\xi(N)\sqrt{\log\log N}),$$
which in turn can be considered as an approximation of the normal order of the function $\tau_N(n)$ over the integers $n\leq N^2$.

When $\alpha$ is such that $\alpha^2t(N)\longrightarrow 0$, as $N\longrightarrow +\infty$, we can clearly Taylor expand the binomial in \eqref{heuristicsummand} and this has been crucial to heuristically estimate the mean of $X_A$.

On the other hand, in the case when $\alpha^2t(N)$ is bounded away from $0$, it is clear that the binomial factor in \eqref{heuristicsummand} can now be considered as ``smaller'' than the other factor $\alpha^2t(N)/2-1.$
In other words, in this range of $\alpha$ we no longer experience cancellation in \eqref{heuristicsummand} due to Taylor expansion, but instead is just the term $\alpha^2t(N)/2-1$ to dominate.

Following these lines of thought, when the limit in the statement of Theorem \ref{mainthm} either does not exist or differs from $-\infty$, we first lower bound the mean value of $X_A$ with:
\begin{align}
\label{heuristicnecessarymean}
\gg \sum_{\substack{1\leq n\leq N ^2\\ \omega(n)\ \approx\ 2\log\log N}}\bigg(\frac{\alpha^2\tau_N(n)}{2}-1\bigg)=\sum_{\substack{1\leq n\leq N ^2\\ \omega(n)\ \approx\ 2\log\log N}}\frac{\alpha^2\tau_N(n)}{2}-\sum_{\substack{1\leq n\leq N ^2\\ \omega(n)\ \approx\ 2\log\log N}}1,
\end{align}
where the relation $\omega(n) \approx 2\log\log N$ will be clarified in a moment, and after show that the first sum on the right hand side above dominates with a contribution of $\gg \alpha^2 N^2$.

On the other hand, if even for these choices of $\alpha$ we have a corresponding random product set in $B(N,\alpha)$ with high probability of maximal size, then by Lemma \ref{propasympmean} this would imply $\mathbb{E}[X_A]=o(\alpha^2 N^2)$. In this way we will reach a contradiction and prove the necessary part in Theorem \ref{mainthm}.

As we said above, in order to precisely lower bound the mean value of $X_A$, we need to carefully determine the approximation $\omega(n) \approx 2\log\log N$ mentioned before. In fact, we will consider integers $n\leq N^2$ with $\omega(n)$ slightly inside the tail of its distribution. Roughly speaking and following the notations introduced before, we will take integers such that:
$$2\log\log N+\xi(N)\sqrt{\log\log N}<\omega(n)\leq 2\log\log N+2\xi(N)\sqrt{\log\log N}.$$
By combining results about the distribution of the prime factors counting function around $\log\log N$ or $2\log\log N$, we will be able to show that the second sum on the right hand side of \eqref{heuristicnecessarymean} makes a negligible contribution compared to the first one there, whereas this last one is seen to be of the same order of the complete sum without any restriction, which by \eqref{averagetau0} contributes $\gg \alpha^2 N^2$.
\section{The sufficient condition}
In this section we are going to prove the sufficient condition in Theorem \ref{mainthm}. To set up the argument, let us suppose that $N\alpha \longrightarrow +\infty$, $\alpha\longrightarrow 0$ and consider a random set $A\in B(N,\alpha)$. We know that we can restrict $\alpha$ in this way thanks to Proposition \ref{equalitycase} and the bound \eqref{multiplicationtablebound}.

Let us then define
$$X_A:=\frac{|A|^2+|A|}{2}-|AA|\geq 0.$$
By \eqref{asympmeanmaxsize} we have
$$\mathbb{E}[X_A]=\frac{\mathbb{E}[|A|^2]}{2}-\mathbb{E}[|AA|]+O(N\alpha ).$$
Our aim is to find conditions on $\alpha$ for which the following holds:

for any $\delta,\eps>0$ there exists an $N_0=N_0(\delta,\eps)$ such that 
$$\mathbb{P}(X_A\geq \delta(|A|^2+|A|)/2)\leq \eps\ \ \ (\text{if}\ N\geq N_0).$$
However, since by \eqref{importcons} $|A|^2+|A|\ \sim |A|^2\ \sim (N\alpha)^2$ with probability $1-o(1)$, we can replace inside the above probability the expression $(|A|^2+|A|)/2$ with just $(N\alpha)^2/2,$ without changing the desired estimate.

By Markov's inequality we have
\begin{equation}
\label{Markovsineq}
\mathbb{P}(X_A\geq \delta(N\alpha)^2/2)\leq  \frac{2\mathbb{E}[X_A]}{\delta(N\alpha)^2}.
\end{equation}
From the proof of \cite[Proposition 3.2]{CRR} we have
\begin{align*}
\mathbb{E}[|AA|] =\sum_{1\leq n\leq N^2}\bigg(1-(1-\alpha^2)^{\tau_N(n)/2}\bigg) + O(N\alpha)
\end{align*}
and by \eqref{asympmeanmaxsize} and \eqref{averagetau0} also that
$$\mathbb{E}[|A|^2]=\sum_{1\leq n\leq N^2}\alpha^2\tau_N(n)+O(N\alpha).$$
Putting the above two identities together we can rewrite the mean of $X_A$ as
\begin{equation}
\label{meanXA}
\mathbb{E}[X_A]=\sum_{1\leq n\leq N^2}\bigg(\frac{\alpha^2\tau_N(n)}{2}-1+(1-\alpha^2)^{\tau_N(n)/2}\bigg) + O(N\alpha).
\end{equation}
Following the heuristic considerations in section $4$, we split the sum in \eqref{meanXA} into two parts, according to the proximity of $\Omega(n)$, which counts the number of prime factors of $n$ with multiplicity, to $2\log\log N$. More specifically, let $M$ be a positive real number that will be chosen at the end as sufficiently large in terms of $\delta,\eps$. We then write
\begin{align}
\label{splittingmeanXA}
\mathbb{E}[X_A]&=\sum_{\substack{1\leq n\leq N^2\\ |\Omega(n)-2\log\log N|\leq M\sqrt{\log\log N}}}\bigg(\frac{\alpha^2\tau_N(n)}{2}-1+(1-\alpha^2)^{\tau_N(n)/2}\bigg) \\
&+\sum_{\substack{1\leq n\leq N^2\\ |\Omega(n)-2\log\log N|> M\sqrt{\log\log N}}}\bigg(\frac{\alpha^2\tau_N(n)}{2}-1+(1-\alpha^2)^{\tau_N(n)/2}\bigg) + O(N\alpha).\nonumber
\end{align}
Since $-1+(1-\alpha^2)^{\tau_N(n)/2}\leq 0$, the second sum above is simply bounded by
\begin{align*}
\ll \alpha^2\sum_{\substack{1\leq n\leq N^2\\ |\Omega(n)-2\log\log N|>M\sqrt{\log\log N}}}\tau_N(n).
\end{align*}
By plugging the definition of $\tau_N(n)$ in we get
\begin{align*}
\sum_{\substack{1\leq n\leq N^2\\ |\Omega(n)-2\log\log N|>M\sqrt{\log\log N}}}\tau_N(n)&=\sum_{\substack{1\leq n\leq N^2\\ |\Omega(n)-2\log\log N|>M\sqrt{\log\log N}}}\sum_{\substack{n/N\leq d\leq N\\ d|n}} 1\\
&=\sum_{d\leq N}\sum_{\substack{n\leq Nd\\ d|n\\ |\Omega(n)-2\log\log N|>M\sqrt{\log\log N}}}1\\
&= \sum_{d\leq N}\sum_{\substack{k\leq N\\ |\Omega(d)+\Omega(k)-2\log\log N|>M\sqrt{\log\log N}}}1\\
&\ll \sum_{\substack{d\leq N\\ |\Omega(d)-\log\log N|>\frac{M}{2}\sqrt{\log\log N}}}\sum_{k\leq N}1\\
&\leq N\sum_{\substack{d\leq N\\ |\Omega(d)-\log\log N|>\frac{M}{2}\sqrt{\log\log N}}}1.
\end{align*}
To compute the last sum above we use a variant of the Erd\H{o}s--Kac's theorem, which states that the result of Proposition \ref{Erdoskac} holds with the function $\Omega(n)$ in place of $\omega(n)$ (and that it easily follows from \cite[ch. III, Theorem 4.15]{T}). It derives an upper bound for the second line in \eqref{splittingmeanXA} of:
\begin{equation}
\label{finalcontributionextomega}
\ll\frac{\alpha^2N^2}{M}\exp\bigg(-\frac{M^2}{8}\bigg)+O(N\alpha),
\end{equation}
thanks to the bound \eqref{consequenceErdoskac}. Clearly, we can make \eqref{finalcontributionextomega} $\leq \delta \eps\alpha^2 N^2/4$, say, if $M=M(\delta,\eps)$ is sufficiently large. Also, note that the upper bound \eqref{finalcontributionextomega} essentially matches our heuristic bound \eqref{tailgauss}, where the constant $M$ here replaces the unbounded function $\lambda(N)$ there.

Overall, we have so far proved that
\begin{equation}
\label{halfmeanXA}
\mathbb{E}[X_A]\leq \sum_{\substack{1\leq n\leq N^2\\ |\Omega(n)-2\log\log N|\leq M\sqrt{\log\log N}}}\bigg(\frac{\alpha^2\tau_N(n)}{2}-1+(1-\alpha^2)^{\tau_N(n)/2}\bigg) +\frac{\delta \eps}{4}\alpha^2 N^2.
\end{equation}
By Lemma \ref{Taylorexpexp} we have
$$(1-\alpha^2)^{\tau_N(n)/2}=\exp\bigg(\frac{\tau_N(n)}{2}\log(1-\alpha^2)\bigg)\leq \exp\bigg(-\frac{\alpha^2\tau_N(n)}{2}\bigg)\leq 1-\frac{\alpha^2\tau_N(n)}{2}+\frac{\alpha^4\tau_N(n)^2}{8},$$
which used in the sum in \eqref{halfmeanXA} gives
\begin{equation}
\label{concsecondmomenttau}
\mathbb{E}[X_A]\leq \sum_{\substack{1\leq n\leq N^2\\ |\Omega(n)-2\log\log N|\leq M\sqrt{\log\log N}}}\frac{\alpha^4\tau_N(n)^2}{8}+\frac{\delta \eps}{4}\alpha^2 N^2.
\end{equation}
Note that the above sum is on the double condition 
$$2\log\log N-M\sqrt{\log\log N}\leq \Omega(n)\leq 2\log\log N+M\sqrt{\log\log N}.$$
By raising both members of the rightmost inequality to the power $2$ and letting $z:=1/2$, we may upper bound the sum in \eqref{concsecondmomenttau} with
\begin{equation}
\label{rankinstrick}
\leq \frac{\alpha^4}{8}2^{2\log\log N+M\sqrt{\log\log N}}\sum_{1\leq n\leq N^2}\tau_N(n)^2 z^{\Omega(n)}.
\end{equation}
Plugging the definition of $\tau_N(n)$ in, we find
$$\sum_{1\leq n\leq N^2}\tau_N(n)^2 z^{\Omega(n)}=\sum_{1\leq n\leq N^2}z^{\Omega(n)}\bigg(\sum_{\substack{d|n\\ n/N\leq d\leq N}}1\bigg)^2.$$
By expanding the square and swapping summations we get the above is
\begin{align}
\label{doublesum}
&=\sum_{1\leq n\leq N^2} z^{\Omega(n)}\sum_{\substack{d_1|n\\ n/N\leq d_1\leq N}}1\sum_{\substack{d_2|n\\ n/N\leq d_2\leq N}}1\nonumber\\
&\ll \sum_{1\leq d_1< d_2\leq N}\sum_{\substack{1\leq n\leq Nd_1\\ n\equiv 0\pmod{[d_1,d_2]}}}z^{\Omega(n)}
+\sum_{1\leq d\leq N}\sum_{\substack{1\leq n\leq Nd\\ n\equiv 0\pmod{d}}}z^{\Omega(n)}.
\end{align}
In the second double sum in \eqref{doublesum} we change variable $n=dk$, with $k\leq N$, to make it
\begin{equation}
\label{contribution1}
=\sum_{1\leq d\leq N}z^{\Omega(d)}\sum_{\substack{1\leq k\leq N}}z^{\Omega(k)}\ll N^2(\log N)^{2z-2}=\frac{N^2}{\log N},
\end{equation}
by two applications of Corollary \ref{boundydiv}.

Regarding the first double sum in \eqref{doublesum} we use the following substitution: $d_1=et_1$, $d_2=et_2$ and $n=t_1t_2e k$. We can then upper bound it with
\begin{align}
\label{splitting12}
\leq \sum_{1\leq e\leq N}z^{\Omega(e)}\sum_{1\leq t_2\leq N/e}z^{\Omega(t_2)}\sum_{1\leq t_1<t_2}z^{\Omega(t_1)}\sum_{k\leq N/t_2}z^{\Omega(k)}.
\end{align}
Notice that the condition $t_1<t_2$ forces $t_2\geq 2$. Moreover, $1\leq N/t_2$ implies $2\leq 2N/t_2$. So, two applications of Corollary \ref{boundydiv} make \eqref{splitting12}
\begin{align*}
&\ll N\sum_{1\leq e\leq N}z^{\Omega(e)}\sum_{2\leq t_2\leq N/e}\frac{z^{\Omega(t_2)}}{t_2}(\log(2N/t_2))^{z-1}\sum_{1\leq t_1<t_2}z^{\Omega(t_1)}\\
&\ll N\sum_{1\leq e\leq N}z^{\Omega(e)}\sum_{2\leq t_2\leq N/e}z^{\Omega(t_2)}(\log(2N/t_2))^{z-1}(\log t_2)^{z-1}.
\end{align*}
By swapping summations and by another application of Corollary \ref{boundydiv} the above is
\begin{align*}
&=N\sum_{2\leq t_2\leq N}z^{\Omega(t_2)}(\log(2N/t_2))^{z-1}(\log t_2)^{z-1}\sum_{e\leq N/t_2}z^{\Omega(e)}\\
&\ll N^2 \sum_{2\leq t_2\leq N}\frac{z^{\Omega(t_2)}}{t_2}(\log(2N/t_2))^{2(z-1)}(\log t_2)^{z-1}\\
&=N^2 \sum_{2\leq t\leq N}\frac{1}{2^{\Omega(t)} t\sqrt{\log t} \log(2N/t)}.
\end{align*}
We now pause a moment to understand the behaviour of the last sum above.
\begin{lem}
\label{lempartialsumm}
For any $N\geq 12$ we have
$$ \sum_{2\leq t\leq N/2}\frac{1}{2^{\Omega(t)} t\sqrt{\log t} \log(N/t)}\ll \frac{\log\log N}{\log N}.$$
\end{lem}
\begin{proof}
To begin with, we split the sum into dyadic intervals:
\begin{align*}
\sum_{2\leq t\leq N/2}\frac{1}{2^{\Omega(t)} t\sqrt{\log t} \log(N/t)}&\leq \sum_{k=1}^{\left\lfloor \frac{\log N}{\log 2}\right\rfloor-2} \sum_{\max\{2,N/2^{k+1}\}< t\leq N/2^k}\frac{1}{2^{\Omega(t)} t\sqrt{\log t} \log(N/t)}\nonumber\\
&\ll \frac{1}{N}\sum_{k=1}^{\left\lfloor \frac{\log N}{\log 2}\right\rfloor-2} \frac{2^k}{k\sqrt{\log(N/2^{k+1})}}\sum_{\max\{2,N/2^{k+1}\}< t\leq N/2^k}\frac{1}{2^{\Omega(t)}}.
\end{align*}
By Corollary \ref{boundydiv} the innermost sum on the second line above is bounded by 
$$\ll \frac{N}{2^k}\frac{1}{\sqrt{\log (N/2^k)}}.$$
Plugging this last estimate in, we find
\begin{align*}
\sum_{2\leq t\leq N/2}\frac{1}{2^{\Omega(t)} t\sqrt{\log t} \log(N/t)}&\ll \sum_{k=1}^{\left\lfloor \frac{\log N}{\log 2}\right\rfloor-2} \frac{1}{k\log(N/2^{k+1})}\\
&\leq\frac{1}{\log(N/4)}+\int_{1}^{\left\lfloor \frac{\log N}{\log 2}\right\rfloor-2}\frac{dt}{t\log(N/2^{t+1})}\\
&=\frac{1}{\log(N/4)}+\frac{\log t-\log\log(N/2^{t+1})}{\log(N/2^{t+1})+t\log 2}\bigg |_{1}^{\left\lfloor \frac{\log N}{\log 2}\right\rfloor-2}\\
&\leq \frac{1}{\log(N/4)}+\frac{\log\log N+O(1)}{\log N+O(1)}+\frac{\log\log(N/4)}{\log(N/4)+\log 2}\\
&\ll \frac{\log\log N}{\log N},
\end{align*}
using that 
$$\bigg\lfloor \frac{\log N}{\log 2}\bigg\rfloor-2=\frac{\log N}{\log 2}+O(1),$$
which proves the lemma.
\end{proof}
With the help of Lemma \ref{lempartialsumm} we can now conclude the estimate of the sum in \eqref{splitting12}, producing for it a bound of 
\begin{equation}
\label{contribution3}
\ll \frac{N^2\log\log N}{\log N}.
\end{equation}
Collecting together \eqref{rankinstrick}, \eqref{contribution1} and \eqref{contribution3}, we have found an overall contribution for the sum in \eqref{concsecondmomenttau} of
\begin{align*}
&\ll\alpha^4 N^2(\log N)^{2\log 2-1}\exp((M\log 2+o(1))\sqrt{\log\log N}).
\end{align*}
Note that it matches our heuristic \eqref{heuristicmeanXA}, where the constant $M$ here replaces the unbounded function $\lambda(N)$ there.

Now suppose that $\alpha$ is such that the quantity:
$$\frac{\log(\alpha^2(\log N)^{\log 4-1})}{\sqrt{\log\log N}}$$
converges as $N\longrightarrow +\infty$ and its limit equals $-\infty$. This is equivalent to say that for any $K>0$ there exists an $N_0=N_0(K)\in \mathbb{N}$ such that for any $N\geq N_0$ we have
$$\alpha^2\leq \frac{1}{(\log N)^{2\log 2-1}\exp(K\sqrt{\log\log N})}.$$
Now, take $K=2M\log 2$ so that the sum in \eqref{concsecondmomenttau} becomes
$$\ll \alpha^2 N^2 \exp((-M\log 2+o(1))\sqrt{\log\log N})$$ 
hence $\leq \delta\eps\alpha^2 N^2/4$, say, if $N$ large enough in terms of $\delta,\eps$. From \eqref{halfmeanXA}, it derives that there exists an $N_0=N_0(\delta,\eps)$ such that for any $N\geq N_0$ we have
$$\mathbb{E}[X_A]\leq \frac{\delta\eps}{2}\alpha^2 N^2.$$
Plugging this into \eqref{Markovsineq} we conclude that
$$\mathbb{P}(X_A\geq \delta(N\alpha)^2/2)\leq \eps,$$
for any $N\geq N_0$, for a sufficiently large $N_0=N_0(\delta,\eps)>0$. This shows the sufficient part in Theorem \ref{mainthm}.
\section{The necessary condition}
In this section we are going to prove the necessary condition in Theorem \ref{mainthm}. 

Let $\alpha\in [0,1)$. We have already noticed that we can confine ourselves with values of $\alpha\longrightarrow 0$ and $N\alpha\longrightarrow +\infty$, thanks to Proposition \ref{equalitycase} and the bound \eqref{multiplicationtablebound}.

Now suppose that we either have that the quantity: 
$$\frac{\log(\alpha^2(\log N)^{\log 4-1})}{\sqrt{\log\log N}}$$
does not converge as $N\longrightarrow +\infty$ or it does, but to a limit different from $-\infty$. 

Then there exists a real number $K$ and a sequence $\{N_k\}_{k\geq 1}$ such that for any $k\geq 1$, we have 
$$\alpha^2\geq \frac{\exp(K\sqrt{\log\log N_k})}{(\log N_k)^{\log 4-1}}.$$
In the following to shorten notations we will indicate with $N$ a generic term of the sequence $N_k$. 

Assume further that even for this choice of $\alpha$ we have a random product set of maximal size, i.e. that $|AA|\ \sim (|A|^2+|A|)/2$ with probability $1-o(1)$, for a random set $A$ in $B(N,\alpha)$.

By Lemma \ref{propasympmean} we deduce that $\mathbb{E}[|AA|]\sim \mathbb{E}[(|A|^2+|A|)/2]$, as $N\longrightarrow +\infty$. Moreover, by the proof of \cite[Proposition 3.2]{CRR} and equations \eqref{asympmeanmaxsize} and \eqref{averagetau0} we can restate this last asymptotic equality as:
\begin{equation}
\label{startingpoint}
\sum_{1\leq n\leq N^2}\bigg(\frac{\alpha^2\tau_N(n)}{2}-1+(1-\alpha^2)^{\tau_N(n)/2}\bigg)=o(N^2\alpha^2).
\end{equation}
The goal is to show that the above sum is larger than a small positive constant times $N^2\alpha^2$, thus contradicting our asymptotic hypothesis for this choice of $\alpha$.

Since by Lemma \ref{Taylorexpexp} we have
$$(1-\alpha^2)^{\tau_N(n)/2}=\exp\bigg(\frac{\tau_N(n)}{2}\log(1-\alpha^2)\bigg)\geq 1+\frac{\tau_N(n)}{2}\log(1-\alpha^2)=1-\frac{\tau_N(n)\alpha^2}{2}+O(\alpha^4 \tau_N(n)),$$
by \eqref{averagetau0} and since $\alpha\longrightarrow 0$, the term inside parenthesis in \eqref{startingpoint} is positive apart from an overall error contribution of $o(\alpha^2 N^2)$. Hence, we can freely discard some unnecessary pieces from the sum to get a lower bound.

In particular, a first lower bound for the sum in \eqref{startingpoint} is given by:
$$\sum_{\substack{1\leq n\leq N^2\\ 2\log\log N+M\sqrt{\log\log N}<\Omega_2(n)\leq 2\log\log N+2M\sqrt{\log\log N}}}\bigg(\frac{\alpha^2\tau_N(n)}{2}-1+(1-\alpha^2)^{\tau_N(n)/2}\bigg),$$
where $M$ is a sufficiently large positive real number that will be chosen later on and we indicate with $\Omega_2(n)$ the function which counts the number of all prime factors of $n$ different form $2$ and counted with multiplicity.

Now, by following our heuristics in subsection 4.6, if we let
\begin{align*}
&h:=2\log\log N+M\sqrt{\log\log N}\\
&m:=2\log\log N+2M\sqrt{\log\log N}
\end{align*}
we further lower bound \eqref{startingpoint} with:
\begin{align}
\label{separatingcontribution}
\sum_{\substack{1\leq n\leq N^2\\ h<\Omega_2(n)\leq m}}\bigg(\frac{\alpha^2\tau_N(n)}{2}-1\bigg)=\frac{\alpha^2}{2}\sum_{\substack{1\leq n\leq N^2\\ h<\Omega_2(n)\leq m}}\tau_N(n)-\sum_{\substack{1\leq n\leq N^2\\ h<\Omega_2(n)\leq m}}1.
\end{align}
The plan is to exhibit a lower bound for the first sum on the right hand side of \eqref{separatingcontribution} and an upper bound for the second one there and compare them. Let us start with the former task. By expanding the definition of $\tau_N(n)$ it is immediate to see that
$$\sum_{\substack{1\leq n\leq N^2\\ h<\Omega_2(n)\leq m}}\tau_N(n)=\sum_{\substack{1\leq n\leq N^2\\ h<\Omega_2(n)\leq m}}\sum_{\substack{d|n\\ n/N\leq d\leq N}}1=\sum_{\substack{1\leq d\leq N}}\sum_{\substack{1\leq k\leq N\\ h<\Omega_2(d)+\Omega_2(k)\leq m}}1,$$
since clearly $\Omega_2(n)$ is still a completely additive function. Moreover, we can lower bound the above with:
$$\geq \sum_{\substack{1\leq d\leq N\\ h/2<\Omega_2(d)\leq m/2}}\sum_{\substack{1\leq k\leq N\\ h/2<\Omega_2(k)\leq m/2}}1=\bigg(\sum_{\substack{1\leq j\leq N\\ h/2<\Omega_2(j)\leq m/2}}1\bigg)^2.$$
To compute the sum into square parenthesis we use a variant of the Erd\H{o}s--Kac's theorem, which states that the result of Proposition \ref{Erdoskac} holds with the function $\Omega_2(n)$ in place of $\omega(n)$ (and that it easily follows from \cite[ch. III, Theorem 4.15]{T}). We deduce that:
$$\sum_{\substack{1\leq j\leq N\\ h/2<\Omega_2(j)\leq m/2}}1=\frac{N}{\sqrt{2\pi}}\int_{M/2}^{M}e^{-t^2/2}dt+ O\bigg(\frac{N}{\sqrt{\log\log N}}\bigg),$$
with a big-Oh constant independent of $M$. 

In conclusion, the first term on the right hand side of \eqref{separatingcontribution} is
\begin{align}
\label{finalcontributionnec1}
&\gg \alpha^2 N^2\bigg(\int_{M/2}^{M}e^{-t^2/2}dt\bigg)^2\\
&\gg_M \frac{N^2}{(\log N)^{\log 4-1}}\exp(K\sqrt{\log\log N}),\nonumber
\end{align}
if $N$ is sufficiently large with respect to $M$ and since $M$ is positive.

On the other hand, we can rewrite the second sum in \eqref{separatingcontribution} as
\begin{equation}
\label{sumPi}
\sum_{h< k\leq m}\Pi(N^2,k),
\end{equation}
where 
$$\Pi(N^2,k):=\sum_{\substack{n\leq N^2\\ \Omega_2(n)=k}}1.$$
Now, we can trivially upper bound $\Pi(N^2,k)$ with 
\begin{align*}
\Pi(N^2,k)\leq \sum_{n\leq N^2}\frac{2^{\Omega_2(n)}}{2^k}\ll \frac{N^2\log N}{2^k},
\end{align*}
thanks to Corollary \ref{boundydiv}, which inserted inside \eqref{sumPi} gives an upper bound for \eqref{sumPi} of:
\begin{align}
\label{finalcontributionnec2}
\ll N^2\log N \sum_{h< k\leq m}\frac{1}{2^k}\ll \frac{N^2\log N}{2^h}\ll \frac{N^2}{(\log N)^{\log 4-1}}\exp((-M\log 2)\sqrt{\log\log N}),
\end{align}
by summing the geometric progression.

By choosing e.g. $M=2|K|/\log 2 +1$, and thanks to \eqref{finalcontributionnec1} and \eqref{finalcontributionnec2}, we have overall showed that
\eqref{separatingcontribution} is
$$\gg_K \alpha^2 N^2,$$
if $N$ large enough in terms of $|K|$, thus contradicting the assertion \eqref{startingpoint} and concluding the proof of the necessary part in Theorem \ref{mainthm}.
\section*{Acknowledgements}
I would like to thank my supervisor Adam J. Harper for some helpful discussions about this problem. Also, I am grateful to Adam and to Carlo Sanna for comments on a first draft of this paper which led to some simplifications in the exposition.

\end{document}